\documentclass[11pt,a4paper,reqno]{amsart}
\usepackage[T1]{fontenc}
\usepackage{lmodern,microtype}
\usepackage[margin=2cm]{geometry}
\usepackage{amsmath,amssymb,mathtools}
\usepackage{enumitem}
\usepackage{hyperref}

\numberwithin{equation}{section}
\newtheorem{theorem}{Theorem}[section]
\newtheorem{proposition}[theorem]{Proposition}
\newtheorem{lemma}[theorem]{Lemma}
\newtheorem{corollary}[theorem]{Corollary}
\theoremstyle{definition}
\newtheorem{definition}[theorem]{Definition}
\theoremstyle{remark}
\newtheorem{remark}[theorem]{Remark}

\newcommand{\R}{\mathbb R}
\newcommand{\T}{\mathbb T}
\newcommand{\dd}{\,\mathrm d}
\newcommand{\Div}{\operatorname{Div}}
\newcommand{\supp}{\operatorname{supp}}
\newcommand{\cO}{\mathcal O}
\newcommand{\cG}{\mathcal G}
\newcommand{\cM}{\mathcal M}
\newcommand{\cK}{\mathcal{K}}
\newcommand{\cP}{\mathcal{P}}
\newcommand{\cV}{\mathcal{V}}
\newcommand{\cH}{\mathcal{H}}
\newcommand{\vav}{\langle v\rangle}
\newcommand{\Tb}{\overline T}
\DeclareMathOperator*{\esssup}{ess\,sup}
\newcommand{\doi}[1]{\href{https://doi.org/#1}{doi:\nolinkurl{#1}}}
\newcommand{\arxiv}[1]{\href{https://arxiv.org/abs/#1}{arXiv:\nolinkurl{#1}}}

\allowdisplaybreaks[2]
\title[Moment Propagation for Vlasov--Poisson]{Moment Propagation for the Vlasov--Poisson System with Gyroscopic Forces, Velocity Diffusion, and Specular Reflection}
\author{Jingpeng Wu}
\address{Department of Applied Mathematics, Nanjing Forestry University, Nanjing 210000, China}
\email{jp.wu@njfu.edu.cn}
\date{}

\begin{document}
\begin{abstract}
This paper studies the three-dimensional Vlasov--Poisson system with gyroscopic forces and velocity diffusion. For the Cauchy problem in $\mathbb R^3$ and the specular--Dirichlet problem in bounded domains, we prove the propagation of velocity moments of order $k>2$ for weak solutions with bounded $L^1\cap L^{\infty}$ norms and kinetic energy. The proof applies Serre's compensated integrability to weighted mass--momentum tensors to derive density estimates. A finite iteration alternating moment propagation with improved density integrability then closes the estimates. These estimates provide a simplified approach to moment propagation and yield  global continuation of classical solutions.
\end{abstract}
\maketitle

\section{Introduction}\label{sec:introduction}

Let $\cO$ be $\R^3$ or a bounded $C^3$ domain $\Omega\subset\R^3$. We consider the
Vlasov--Poisson system
\begin{equation}\label{eq:kinetic}
 \begin{aligned}
 &\partial_t f+v\cdot\nabla_x f
  +\nabla_v\cdot\bigl((E+\cG-\lambda v)f\bigr)
  =\nu\Delta_v f,\\
 &-\Delta U=4\pi\rho,
  \qquad E=-\gamma\nabla U.
 \end{aligned}
\end{equation}
Here $x\in\cO$, $v\in\R^3$, and $f\ge0$ is the particle distribution, with spatial density
$\rho(t,x)=\int_{\R^3}f(t,x,v)\dd v$. The parameters satisfy $\nu,\lambda\ge0$, while $\gamma=1$ and
$\gamma=-1$ correspond to repulsion and attraction, respectively.

For $\cO=\Omega$, let $n$ denote the outward unit normal. We impose specular reflection
\begin{equation}\label{eq:reflection}
 f(t,x,v)=f(t,x,R_xv),\qquad
 R_xv=v-2(v\cdot n(x))n(x),\qquad x\in\partial\Omega.
\end{equation}
The potential satisfies
\begin{equation}\label{eq:poisson}
 \begin{split}
  \lim_{|x|\to\infty}U(t,x)=0,\qquad
       &\cO=\R^3,\\
  U|_{\partial\Omega}=0,
       \qquad&\cO=\Omega.
 \end{split}
\end{equation}
The conditions \eqref{eq:reflection} and  \eqref{eq:poisson} define the specular--Dirichlet problem. With $\vav=(1+|v|^2)^{1/2}$, the prescribed gyroscopic force satisfies
\begin{equation}\label{eq:gyro}
 v\cdot\cG(t,x,v)=0,\qquad \operatorname{div}_v\cG=0,
 \qquad |\cG(t,x,v)|\le b(t)\vav^2,
 \qquad b\in L^1_{\rm loc}([0,\infty)).
\end{equation}

\subsection{Notation}
For $r\ge0$, define
\begin{equation}\label{eq:moment-def}
  m_r(t,x)=\int_{\R^3}\vav^r f(t,x,v)\dd v,
 \qquad M_r(t)=\int_\cO m_r(t,x)\dd x.
\end{equation}

\begin{definition}[Weak solutions]\label{def:weak-solution}
	Let $0\le f_0\in L^1\cap L^{\infty}$ and
	$M_2(0)<\infty$. Let $\mathcal X^{\gamma}$ consist of weak
	solutions $f\ge 0$ of~\eqref{eq:kinetic} with initial datum $f_0$,
	potential condition~\eqref{eq:poisson}, and specular reflection~\eqref{eq:reflection}
	when $\cO=\Omega$, such that, for every $T>0$
	\begin{equation}\label{eq:base-bounds}
		\begin{gathered}
			f\in L^\infty\bigl(0,T;L^1(\cO\times\R^3)\cap L^\infty(\cO\times\R^3)\bigr),
			\qquad \esssup_{0<t<T} M_2(t)<\infty,\\
			Ef\in L^1_{\mathrm{loc}}\bigl([0,T)\times\overline{\cO}\times\R^3\bigr).
		\end{gathered}
	\end{equation}
	Explicitly, for every admissible test function $\psi\in C^{\infty}([0,T)\times\cO\times\R^3)$,
	\begin{equation}\label{eq:weak}
		\int_0^T\!\iint f\left(
		\partial_t\psi+v\cdot\nabla_x\psi
		+(E+\cG-\lambda v)\cdot\nabla_v\psi+\nu\Delta_v\psi\right)
		+\iint f_0\psi(0)=0.
	\end{equation}
	Admissible tests are compactly supported when $\cO=\R^3$. In a bounded domain they are compactly supported in $t,v$, smooth up to the
	spatial boundary and satisfy $\psi(t,x,R_xv)=\psi(t,x,v)$ on $\partial\Omega$.
\end{definition}

For classical continuation, we work on an existence interval $[0,T_*)$, with $T_*\le\infty$. Fix \(\overline T<\infty\) and let \(0<T<\min\{T_*,\overline T\}\).

The constant $C$ may depend on
$\|f_0\|_1$, $\|f_0\|_\infty$, $M_2(0)$, the fixed geometry of $\Omega$, $\nu,\lambda$, $\int_0^T b(t)\dd t$. In the proof of Theorem~\ref{thm:classical}, it may also depend on $\|\nabla\cG\|_{\infty}$ or $\sup_{0\le t\le T}\|A(t)\|$. Additional dependence on $a$ denotes $C_a$. In particular, $C_{T,\ldots}$ denotes a bound chosen
continuous in $T$ on every interval. The same convention applies to $C_{\overline{T},\ldots}$.

\subsection{Main results}

\begin{theorem}[Density and velocity moments]\label{thm:moments}
Let $f\in\mathcal X^{\gamma}$. Then
\begin{equation}\label{eq:critical-main}
 \int_0^T\|\rho(t)\|_2^2\dd t\le C_{T}.
\end{equation}
For every $k>2$ with $M_k(0)<\infty$, set $p(k)=(4k+10)/9$. Then
\begin{equation}\label{eq:moment-main}
 \esssup_{0<t<T}M_k(t)
       +\int_0^T\!\int_\cO\rho(t,x)^{p(k)}\dd x\dd t\le C_{T,k,M_k(0)}.
\end{equation}
In particular, if $M_5(0)<\infty$, then
\begin{equation}\label{eq:field-main}
 \|E\|_{L^{10/3}(0,T;L^\infty(\cO))}\le C_{T,M_5(0)}.
\end{equation}
\end{theorem}

\begin{theorem}[Classical solutions in the whole space]\label{thm:classical}
Let $\cO=\R^3$, $\lambda\ge0$ and $\gamma\in\{-1,1\}$.
\begin{enumerate}[label=\textup{(\roman*)},leftmargin=2em]
\item Suppose $\nu=0$, and $\cG$ satisfies~\eqref{eq:gyro}. Assume that $\cG$ is $C^1$,
with first derivatives bounded on compact subsets of phase space over finite time intervals. Every
nonnegative $f_0\in C_c^1(\R^3\times\R^3)$ generates a unique global classical solution $f$ of \eqref{eq:kinetic}. For every $T>0$, $f\in C_c^1([0,T]\times\R^3\times\R^3)$.
\item Suppose $\nu>0$, and $\cG(t,v)=A(t)v$, where $A(t)$ is a smooth skew-symmetric
matrix. Nonnegative Schwartz data on $\R^6$ generate a unique global classical solution. All polynomial
velocity moments and all polynomially weighted derivatives in $(x,v)$ remain bounded on finite time
intervals.
\end{enumerate}
\end{theorem}

\begin{theorem}[Classical solutions for the specular--Dirichlet problem]\label{thm:boundary}
Let $\Omega$ be a bounded $C^3$ convex domain, and let
$\gamma=1$, $\nu=\lambda=0$, and $\cG=0$. Suppose that $f_0\ge0$ belongs to the $C^{1,\alpha}$
specular class in~\cite[Theorem 2.1]{Cesbron}, for some $0<\alpha<1$. Then the specular--Dirichlet
problem has a unique global classical solution in that class. Its velocity support is bounded on finite intervals.
\end{theorem}

\begin{remark}[Other geometries and boundary conditions]\label{rem:other-settings}
	The density and moment estimates also extend to the periodic problem on $\T^3$, to the
	specular--Neumann and absorbing--Dirichlet problems. Absorbing--Neumann problems additionally require compatible charge balance and
	control of boundary energy transfer. For brevity, we develop only the Cauchy and specular--Dirichlet
	problems here.
\end{remark}

\subsection{Historical context}\label{subsec:history}
The three-dimensional global classical theory of the Vlasov--Poisson system rests on two approaches: the characteristic estimates of
Pfaffelmoser ~\cite{Pfaffelmoser} and Schaeffer's simplified version~\cite{Schaeffer}, and the velocity-moment estimates of Lions and
Perthame~\cite{LP}. Pallard subsequently obtained further moment propagation results~\cite{Pallard}.

For equations with velocity diffusion, Perthame developed higher-moment estimates~\cite{Perthame90}, and
Bouchut established the three-dimensional global smooth theory of
Vlasov--Poisson--Fokker--Planck~\cite{Bouchut93,Bouchut95}. Pulvirenti and Simeoni subsequently obtained
density estimates by a probabilistic
argument~\cite{PulvirentiSimeoni}. Prescribed magnetic fields introduce a different difficulty. Rege studied moment propagation for uniform and
time-dependent magnetic fields~\cite{Rege21,Rege23}. For spatially varying fields, Ben Porat, Gagnebin,
Iacobelli, and Junn\'e considered the two-dimensional Poisson and three-dimensional screened Poisson
equations~\cite{BGIJ}. Here we consider the three-dimensional unscreened Poisson
equation.

In bounded convex domains, Hwang and Vel\'azquez, Hwang, Jung, and Vel\'azquez, and Cesbron and Iacobelli
constructed classical solutions for specular--Dirichlet and specular--Neumann problems~\cite{HV,HJV,Cesbron}. Fern\'andez-Real
studied their weak formulation and Lagrangian structure~\cite{FR}. The corresponding moment propagation results were not established in these works.

Serre's compensated integrability inequalities apply to positive symmetric tensors with measure-valued divergence~\cite{Serre18,Serre19,SerreBoundary}. Recent kinetic applications include Wang's work on
energy conservation and moment propagation for Vlasov--Poisson~\cite{Wang} and Katriadakis's global
solvability results for higher-dimensional Vlasov equations with power-law interactions~\cite{Kat}. Here
radial velocity weights allow the same tensor estimate to accommodate gyroscopic forces, velocity
diffusion, and specular reflection.

\subsection{Proof strategy}\label{subsec:strategy}
The proof is based on the radially weighted mass and momentum tensor
\begin{equation}\label{eq:intro-tensor}
 A_r(t,x)=\int_{\R^3}\vav^r
              \binom{1}{v}\otimes\binom{1}{v} f(t,x,v)\dd v,
 \qquad r\ge0.
\end{equation}
A lower bound for its determinant converts compensated integrability into an estimate for a higher power
of $\rho$. The tensor and its source involve velocity moments through order $r+2$: the gyroscopic force
annihilates the radial weight, while the velocity Laplacian lowers its order. Specular reflection also
preserves the weight and cancels the tangential trace in Serre's cylinder inequality. Consequently, the
same weighted estimate holds in the whole space and in the bounded domain. At $r=0$, the $\|f\|_{1}$, $\|f\|_{\infty}$ and kinetic energy bounds control the electric source and give $\rho\in L^2_{t,x}$.

Starting from this $L^2$ estimate, we alternate moment propagation with density estimates derived from weighted tensors.
If $M_5(0)<\infty$, the weights $1$, $\vav$, and $\vav^3$ give the finite sequence
\begin{equation}\label{eq:intro-iteration}
 \rho\in L^2_{t,x}
 \ \Longrightarrow\ M_3
 \ \Longrightarrow\ \rho\in L^{22/9}_{t,x}
 \ \Longrightarrow\ M_5
 \ \Longrightarrow\ \rho\in L^{10/3}_{t,x}.
\end{equation}
Since $10/3>3$, the final density bound yields $E\in L^1_tL^\infty_x$, which propagates every higher
moment that is initially finite. For target orders below five, the iteration stops at that order and
uses only lower initial moments. The exponent $p(k)$ in Theorem~\ref{thm:moments} is the density gain
produced by the weight $\vav^{k-2}$.

For classical solutions without diffusion, the initial $L^2$ density estimate already suffices to
control the velocity support. Indeed, the whole-space Coulomb estimate is
\begin{equation}\label{eq:intro-field}
 \|E(t)\|_\infty
 \le C\|\rho(t)\|_2^{2/3}\|\rho(t)\|_\infty^{1/3}.
\end{equation}
If $f$ is a classical solution supported in $|v|\le\cV(t)$, then
$\|\rho(t)\|_\infty\le C_{\Tb}\cV(t)^3$. Thus the right-hand side is bounded by
$C_{\Tb}\|\rho(t)\|_2^{2/3}\cV(t)$. In the bounded domain, the regular part of
the kernel adds a mass term; see~\eqref{eq:field-split}. Since $\|\rho\|_2^{2/3}$ is integrable in
time, Gr\"onwall's inequality bounds $\cV$. The same argument applies along free and reflected
trajectories, because the gyroscopic force does no work and reflection preserves speed. The local
regularity theory then gives classical continuation.

With velocity diffusion, weighted uniform estimates replace compact velocity support. The fifth-moment
estimate first gives $E\in L^1_tL^\infty_x$; a weighted maximum principle then yields polynomial
velocity decay and a uniform bound for $\rho$. Finally, logarithmic Poisson estimates and the
differentiated kinetic equation control the classical norms. In the diffusive classical theorem, the
drift $A(t)v$ allows differentiation at a fixed velocity weight, whereas the moment theorem applies to
the more general forces in~\eqref{eq:gyro}.

\subsection{Organization}\label{subsec:organization}
Section~\ref{sec:density} proves the low-order bounds and the density estimate obtained from the weighted tensor.
Section~\ref{sec:moments} combines them to prove Theorem~\ref{thm:moments}. Section~\ref{sec:classical}
then establishes the classical results for the Cauchy problem and the specular--Dirichlet problem.

\subsection*{Acknowledgements}
The research was supported by the National Natural Science Foundation of China (Grant No.\ 12501282).

\subsection*{Declaration of AI Use}
During the preparation of this manuscript, the author used OpenAI's ChatGPT (Sol 5.6 and Astra 6) to assist with mathematical discussions, preliminary examination of arguments, language editing, and LaTeX formatting. The author takes full responsibility for the mathematical correctness, originality, and final content of the manuscript.

\section{Density estimates}\label{sec:density}

\subsection{Energy and the first moments}
We begin with the low-order estimates needed to control the tensor divergence. 

For both potential problems in~\eqref{eq:poisson}, the Hardy--Littlewood--Sobolev inequality or elliptic
regularity gives
\begin{equation}\label{eq:poisson-estimates}
 \|E\|_{3s/(3-s)}\le C_s\|\rho\|_s,\qquad1<s<3.
\end{equation}
The corresponding Green function satisfies
\begin{equation}\label{eq:green}
 |\nabla_x G(x,y)|\le C\bigl(1+|x-y|^{-2}\bigr).
\end{equation}
On $\R^3$, the bounded remainder can be omitted. We refer to~\cite{LiebLoss,GT} for these
potential and elliptic estimates.

For a classical solution, define the kinetic and potential energies by
\begin{equation}\label{eq:energy}
 \cK(t)=\frac12\iint |v|^2f,\qquad
 \cP(t)=\frac1{8\pi}\int_\cO|\nabla U|^2
     =\frac12\int_\cO\rho U,\qquad
 \cH_\gamma=\cK+\gamma\cP.
\end{equation}
The identity for $\cP$ follows by integration by parts, using the decay of $U$ or the homogeneous
Dirichlet condition. In particular, $\cP\ge0$.

We recall the following two standard estimates, see e.g.~\cite{LP,Bouchut93}.
\begin{lemma}\label{lem:energy}
Let $f\in\mathcal{X}^{\gamma}$. Then
\begin{equation}\label{eq:low-bounds}
 \begin{gathered}
  \esssup_{t<T}\|f(t)\|_1\le\|f_0\|_1,\qquad \esssup_{t<T}\|f(t)\|_\infty\le e^{3\lambda T}\|f_0\|_\infty,\\
  \esssup_{t<T}\bigl(\cK(t)+\cP(t)\bigr)\le C_{\Tb}
 \end{gathered}
\end{equation}

\end{lemma}

\begin{lemma}\label{lem:kinetic-interpolation}
	Let $f\in\mathcal{X}^{\gamma}$. Then
\begin{equation}\label{eq:kinetic-interpolation}
	\| m_r\|_{(k+3)/(r+3)}
	\le C_{k,r}\|f(t)\|_\infty^{(k-r)/(k+3)}M_k^{(r+3)/(k+3)},
	\qquad 0\le r<k,
\end{equation}
\end{lemma}

\begin{remark}\label{rem:preliminary-force}
By Lemma~\ref{lem:energy}, we have also
\begin{equation}\label{eq:first-source}
 \esssup_{t<T}\left(\|m_1(t)\|_1+\|E(t)\|_2
          +\int_\cO\rho(t,x)|E(t,x)|\dd x\right)\le C_{T}.
\end{equation}
In addition,~\eqref{eq:gyro} gives
\begin{equation}\label{eq:gyro-low}
 \iint|\cG f|\dd v\dd x\le b(t)M_2(t).
\end{equation}
\end{remark}

\subsection{The tensor and the boundary trace}\label{subsec:tensor-trace}
Let $S$ be an integrable symmetric positive semidefinite tensor on $\R^4$, and suppose that its row
divergence $\Div S$ is a finite vector-valued Radon measure. Serre's inequality states that
\begin{equation}\label{eq:serre-whole}
 \int_{\R^4}(\det S)^{1/3}\le C\|\Div S\|_{\cM}^{4/3},
\end{equation}
where $\|\cdot\|_{\cM}$ denotes total variation. The compactly supported version is~\cite[Theorem
2.2]{Serre19}; the integrable case follows by spatial cutoffs. On the cylinder $\R\times\Omega$, we use
\begin{equation}\label{eq:serre-cylinder}
 \int_{\R\times\Omega}(\det S)^{1/3}
 \le C\left(
 \|\Div S\|_{\cM}+\|\gamma_N^\tau S\|_{\cM}+\|S\|_1
 \right)^{4/3}.
\end{equation}
Here $N=(0,n)$ is the lateral unit normal and $\gamma_N^\tau S$ is the tangential component of the
normal trace. This form follows from~\cite[Theorem 1.2]{SerreBoundary} by bounding its geometric column
norm by $C\|S\|_1$.

\begin{lemma}[Vanishing tangential trace]\label{lem:specular-trace}
Let $f\in\mathcal X^\gamma$ with $\cO=\Omega$, and let
$w\in C_c^\infty(\R^3)$ be radial. Set
$A_w=\int w(1,v)\otimes(1,v)f\dd v$. Then
\begin{equation}\label{eq:trace-zero}
 \gamma_N^\tau A_w=0\qquad\text{on }(0,T)\times\partial\Omega.
\end{equation}
\end{lemma}
\begin{proof}
For a smooth solution, write $v_\tau=v-(v\cdot n)n$. Since reflection preserves
$w$ and $v_\tau$ while reversing $v\cdot n$,
\[
 \gamma_N^\tau A_w
 =\binom{\displaystyle\int w(v)(v\cdot n)f\dd v}
        {\displaystyle\int w(v)v_\tau(v\cdot n)f\dd v}
 =0.
\]
The remaining normal component is
\begin{equation*}
 A_wN=\binom{0}{\displaystyle n\int w(v)(v\cdot n)^2f\dd v}.
\end{equation*}
This normal pressure is the boundary contribution associated with extension by zero.
The cylinder inequality requires only the tangential component~\eqref{eq:trace-zero}.

For a weak solution, take a $C^1$ vector test $\zeta=(\zeta_0,\zeta_x)$ on
$(0,T)\times\overline\Omega$, compactly supported in time, with $\zeta_x\cdot n=0$
on $\partial\Omega$. The test
\[
 w(v)\bigl(\zeta_0+v\cdot\zeta_x\bigr)
\]
is reflection invariant. Writing $z=(1,v)$ and
\[
 Q_w=\int f\bigl[(E+\cG-\lambda v)\cdot\nabla_v(wz)
                         +\nu\Delta_v(wz)\bigr]\dd v,
\]
the weak formulation gives
\[
 \int_0^T\!\int_\Omega A_w:\nabla_{t,x}\zeta
       +\int_0^T\!\int_\Omega Q_w\cdot\zeta=0.
\]
Since this holds for every tangential vector test, its normal trace has
zero tangential component, proving~\eqref{eq:trace-zero}.
\end{proof}

\begin{remark}[Velocity diffusion and the specular trace]\label{rem:diffusion-trace}
The cancellation also holds with velocity diffusion. At each boundary point, $R_x$ is orthogonal, and
therefore
\begin{equation*}
 \Delta_v(\varphi\circ R_x)=(\Delta_v\varphi)\circ R_x.
\end{equation*}
Thus $\Delta_v$ preserves reflection invariance. Its contribution to the weak identity is the volume
source $\nu\int f\Delta_v(w\chi_R(\zeta_0+v\cdot\zeta_x))\dd v$, so the preceding trace calculation is
unchanged for $\nu>0$.
\end{remark}

\begin{remark}[Absorbing boundaries]\label{rem:absorbing-ci}
	For a solution
	satisfying the kinetic Green formula with zero incoming trace, let $f_+$
	denote the outgoing trace.  Unlike specular reflection, absorption
	generally gives a nonzero tangential trace. Nevertheless, for $r\ge0$,
	\[
	\begin{aligned}
		\|\gamma_N^\tau A_r\|_{\cM}
		&\le \int_0^T\!\int_{\partial\Omega}\!
		\int_{v\cdot n>0}
		\vav^{r+1}(v\cdot n)f_+\dd v\dd S_x\dd t\\
		&\le (1+C_r\nu T)\esssup_{0<t<T}M_{r+1}(t)
		+C_r\int_0^T\!\int_\Omega |E|m_r\dd x\dd t,
	\end{aligned}
	\]
	The last inequality follows from the
	radial moment balance with weight $\vav^{r+1}$: the gyroscopic term
	vanishes, friction contributes nonpositively, and diffusion involves
	only lower-order moments. Thus the outgoing flux is controlled by the
	same quantities as the weighted tensor estimate, and
	\eqref{eq:serre-cylinder} remains applicable.
\end{remark}

\subsection{Density estimates via weighted tensors}
We first relate the determinant of the weighted tensor to the density. We then estimate its divergence
and apply compensated integrability, using the preceding trace cancellation in the bounded domain.

\begin{lemma}[The determinant]\label{lem:determinant}
Let $r\ge0$, $F>0$, and $0\le g\le F$. Suppose that $\vav^{r+2}g\in L^1(\R^3)$, and write
$m=\int_{\R^3}g$. Then
\begin{equation}\label{eq:determinant}
 \det\left(\int_{\R^3}\vav^r
       \binom{1}{v}\otimes\binom{1}{v}g(v)\dd v\right)
 \ge c_r F^{-2-4r/3}m^{6+4r/3}.
\end{equation}
\end{lemma}
\begin{proof}
The conclusion is immediate when $m=0$. For $m>0$, first consider the unweighted tensor and define
\[
 u=\frac1m\int vg\dd v,\qquad
 \Sigma=\frac1m\int(v-u)\otimes(v-u)g\dd v.
\]
The covariance matrix $\Sigma$ is positive definite, since a nonzero Lebesgue density cannot be
supported on an affine plane. Under the probability measure $g\dd v/m$, the quadratic form
$(v-u)^T\Sigma^{-1}(v-u)$ has mean three. At least half of the mass therefore lies in the ellipsoid
where this form is at most six. Its volume is $C\sqrt{\det\Sigma}$, so $m/2\le CF\sqrt{\det\Sigma}$. The
Schur complement then gives
\[
 \det A_0[g]=m^4\det\Sigma\ge cF^{-2}m^6.
\]

For the weighted tensor, choose $R>0$ so that $F|B_R|=m/2$, and set $g_o=g\mathbf1_{|v|\ge R}$. Then
$\int g_o\ge m/2$ and $A_r[g]\ge R^rA_0[g_o]$ in the order of positive semidefinite matrices. Applying
determinant monotonicity and the unweighted estimate, with $R^3=cm/F$, yields
\[
 \det A_r[g]\ge cR^{4r}F^{-2}m^6
   =c_r F^{-2-4r/3}m^{6+4r/3}.\qedhere
\]
\end{proof}

\begin{proposition}[Density estimate via weighted tensors]\label{prop:tensor}
Let $f\in\mathcal X^{\gamma}$. For every $r\ge0$ such
that the right-hand side is finite,
\begin{equation}\label{eq:tensor-bound}
\begin{split}
 \int_0^T\!\int_\cO\rho^{2+4r/9}\dd x\dd t\le& C_{T,r}
 \biggl\{\int_0^T\!\int_\cO|E|m_r\dd x\dd t+\esssup_{t<T}M_{r+1}(t)\\
 &\qquad+\int_0^T(1+\nu+\lambda+b(t))M_{r+2}(t)\dd t\biggr\}^{4/3}.
\end{split}
\end{equation}
\end{proposition}
\begin{proof}
Set $z=(1,v)$ and $w_r=\vav^r$. Testing the kinetic equation in the interior with $w_rz_i$ gives
\begin{equation}\label{eq:tensor-div}
 (\Div A_r)_i
 =\int f\left[(E+\cG-\lambda v)\cdot\nabla_v(w_rz_i)
                       +\nu\Delta_v(w_rz_i)\right]\dd v.
\end{equation}
Since $\cG\cdot\nabla_vw_r=0$,
\begin{equation}\label{eq:tensor-source}
 |\Div A_r|
 \le C_r\left[|E| m_r+(\nu+\lambda) m_{r+1}
                           +b(t) m_{r+2}\right].
\end{equation}
The tensor itself satisfies $|A_r|\le Cm_{r+2}$ and $|A_re_0|\le Cm_{r+1}$. To verify the diffusion
contribution, note that
\begin{equation}\label{eq:diffusion-weights}
 \Delta_v w_r=r\bigl[3+(r+1)|v|^2\bigr]\vav^{r-4},
 \qquad \Delta_v(w_r v_i)=v_i\Delta_vw_r+2r v_i\vav^{r-2}.
\end{equation}
Thus diffusion vanishes in the unweighted mass and momentum balances and contributes only lower-order
moments for $r>0$.

To justify~\eqref{eq:tensor-div}, first test with $w_rz_i\chi_R$, where $\chi_R$ is radial. The electric
cutoff terms are bounded by $C_r|E|m_r$, the friction and diffusion terms by $C_r(\nu+\lambda)m_{r+2}$,
and the gyroscopic terms by $C_rb\,m_{r+2}$. Moreover, $\cG\cdot\nabla_v\chi_R=0$. Dominated convergence
therefore gives convergence of the tensors and their divergences in $L^1$.

For each $R$, Lemma~\ref{lem:specular-trace} gives
$\gamma_N^\tau A_{w_r\chi_R}=0$. The $L^1$ convergence of the tensors and
their divergences allows passage to the limit in the distributional normal
traces, yielding $\gamma_N^\tau A_r=0$. Thus the moment order $r+2$
controls both the tensor and its gyroscopic source, while the time trace
requires only $M_{r+1}$.

Next choose $\eta_\varepsilon\in C_c^\infty(0,T)$ with $0\le\eta_\varepsilon\le1$,
$\eta_\varepsilon\to1$ in $(0,T)$, and $\int|\eta_\varepsilon'|\le2$, and extend $\eta_\varepsilon A_r$
by zero in time. Then
\[
 \Div(\eta_\varepsilon A_r)
 =\eta_\varepsilon\Div A_r+\eta_\varepsilon'A_re_0.
\]
The first term is bounded by~\eqref{eq:tensor-source}, and the second by $\esssup M_{r+1}$. We may thus
apply~\eqref{eq:serre-whole} on $\R\times\R^3$, or~\eqref{eq:serre-cylinder} on
$\R\times\Omega$ with zero tangential trace. The geometric term is bounded by
$C\int_0^T M_{r+2}$. On $\R^3$, the spatial cutoff error is at most $CR^{-1}\int|A_r|$ and tends to
zero.

Finally,~\eqref{eq:determinant} gives
\[
 (\det A_r)^{1/3}\ge c_r\|f\|_{L^\infty((0,T)\times\cO\times\R^3)}^{-2/3-4r/9}\rho^{2+4r/9}.
\]
Fatou's lemma removes the time cutoffs and gives~\eqref{eq:tensor-bound}.
\end{proof}

\begin{corollary}[The initial density estimate]\label{cor:critical}
For $f\in\mathcal X^{\gamma}$,
\begin{equation}\label{eq:rho2-E6}
 \int_0^T\left(\|\rho(t)\|_2^2+\|E(t)\|_6^2\right)\dd t\le C_{T}.
\end{equation}
\end{corollary}
\begin{proof}
Apply Proposition~\ref{prop:tensor} with $r=0$. All terms on its right-hand side are controlled
by~\eqref{eq:base-bounds},~\eqref{eq:first-source}, and~\eqref{eq:gyro-low}. This proves the density
bound. The field bound follows from~\eqref{eq:poisson-estimates} with $s=2$.
\end{proof}

\section{Propagation of velocity moments}\label{sec:moments}

Alternating moment propagation with improved density integrability proves Theorem~\ref{thm:moments}.

\subsection{A sufficient density estimate}
For a target order $k>2$, define
\begin{equation}\label{eq:sk}
 s(k)=\frac{3(k+3)}{k+6}<3.
\end{equation}
\begin{lemma}[A single moment]\label{lem:single-moment}
Let $f\in\mathcal X^{\gamma}$ with $M_k(0)<\infty$. If
\begin{equation}\label{eq:density-condition}
 \int_0^T\|\rho(t)\|_{s(k)}\dd t<\infty,
\end{equation}
then $\esssup_{t<T}M_k(t)<\infty$.
\end{lemma}
\begin{proof}
Assume temporarily that the full moment is finite. Testing with $\vav^k$ gives
\begin{equation*}
 \frac{\dd}{\dd t}M_k
 \le C_k\int_\cO|E| m_{k-1}+C_k\nu M_{k-2}.
\end{equation*}
The radial weight cancels the gyroscopic term and the reflected boundary flux. Friction contributes
nonpositively, while~\eqref{eq:diffusion-weights} gives $\Delta_v\vav^k\le C_k\vav^{k-2}$.

To estimate the electric source, observe that $3s(k)/(3-s(k))=k+3$. The Poisson
estimate~\eqref{eq:poisson-estimates} and kinetic interpolation~\eqref{eq:kinetic-interpolation} yield
\begin{equation*}
 \int_\cO|E|m_{k-1}
 \le C_{T,k}\|\rho\|_{s(k)}M_k^{(k+2)/(k+3)}.
\end{equation*}
Moreover, $M_{k-2}\le \|f_0\|_1^{2/k}M_k^{(k-2)/k}$ and $(k-2)/k<(k+2)/(k+3)$. Dividing the moment inequality
by the appropriate power of $1+M_k$ therefore gives
\begin{equation}\label{eq:moment-root}
 \frac{\dd}{\dd t}(1+M_k)^{1/(k+3)}
 \le C_{T,k}(1+\|\rho(t)\|_{s(k)}).
\end{equation}

To justify the estimate without assuming finiteness of the full moment, choose a smooth nondecreasing
concave function $\theta$ on $[0,\infty)$ with $\theta(s)=s$ for $s\le1$ and $\theta'(s)=0$ for
$s\ge2$. For $R\ge2$, set
\[
 \Phi_R(v)=R^k\theta(\vav^k/R^k),\qquad
 Y_R(t)=\iint\Phi_R(v)f(t,x,v).
\]
Concavity gives $\Phi_R\uparrow\vav^k$. On the support of its derivatives,
\[
 \Phi_R\ge c_k\vav^k,\quad
 |\nabla_v\Phi_R|\le C_k\vav^{k-1},\quad
 |\Delta_v\Phi_R|\le C_k\vav^{k-2},\quad
 v\cdot\nabla_v\Phi_R\ge0.
\]
On this set, the local $k$th moment is bounded by $C_k\int\Phi_Rf\dd v$. Hence the preceding
interpolation applies with $Y_R$ in place of $M_k$, while diffusion contributes at most
$C_k\|f_0\|_1^{2/k}Y_R^{(k-2)/k}$. We obtain~\eqref{eq:moment-root} for $Y_R$, uniformly in $R$.

For fixed $R$, the bounded weight $\Phi_R$ is admitted in~\eqref{eq:weak} by first multiplying it by a
radial cutoff $\chi_S$ and then letting $S\to\infty$. The electric errors are dominated by $C_{k,R}\rho|E|$,
which is integrable by~\eqref{eq:first-source}; the friction and diffusion errors tend to zero with the
velocity tails, and the gyroscopic cutoff term vanishes. On $\R^3$, spatial cutoff errors also vanish,
since the cutoff gradients tend to zero and $\iint|v|f$ is bounded. In the bounded domain, the tests
remain reflection invariant.

We may therefore integrate the uniform root inequality and let $R\to\infty$. Monotone convergence
then bounds $M_k$ and proves the lemma.
\end{proof}

\subsection{From moments back to density}
The next lemma supplies the reverse implication needed for the iteration.

\begin{lemma}[The weighted source]\label{lem:electric-source}
Let $f\in\mathcal X^{\gamma}$ satisfy~\eqref{eq:rho2-E6}. If
$\esssup_{t<T}M_m(t)<\infty$ for some
$2<m\le5$, then
\begin{equation}\label{eq:density-from-moment}
 \int_0^T\!\int_\cO\rho^{(4m+10)/9}\dd x\dd t
 \le C_{T,m}\left(1+\esssup_{t<T}M_m(t)\right)^{4/3}.
\end{equation}
For $m>5$ with $\esssup_{t<T}M_m(t)<\infty$ and $E\in L^1(0,T;L^\infty)$,
\[
 \int_0^T\!\int_\cO\rho^{(4m+10)/9}\dd x\dd t
 \le C_{T,m}\left[(1+\esssup_{t<T}M_m(t))
       \left(1+\int_0^T\|E(t)\|_\infty\dd t\right)\right]^{4/3}.
\]
\end{lemma}
\begin{proof}
Set $r=m-2$. H\"older's inequality in velocity and then in space gives
\begin{equation*}
 \int_\cO |E| m_r
 \le M_m^{r/m}\left(\int_\cO\rho|E|^{m/2}\right)^{2/m}.
\end{equation*}
For $2<m\le5$, apply spatial H\"older and interpolate $E\in L^m_x$ between the bounds in $L^2_x$ and
$L^6_x$. Using~\eqref{eq:rho2-E6}, we obtain
\begin{equation*}
 \begin{aligned}
 \int_\cO|E|m_r
 &\le M_m^{r/m}\|\rho\|_2^{2/m}\|E\|_m\\
 &\le C_{T,m}M_m^{r/m}(1+\|\rho\|_2)^{1+2/m}.
 \end{aligned}
\end{equation*}
Since $1+2/m<2$, Corollary~\ref{cor:critical} controls its time integral.
Proposition~\ref{prop:tensor} then gives~\eqref{eq:density-from-moment}, because all moment
powers inside its braces are at most one. For general $m>2$, use
$\int|E|m_r\le\|E\|_\infty M_r$ in the same proposition.
\end{proof}

\begin{remark}[From density integrability to a bounded field]\label{rem:large-density-field}
For $p>3$, the singular part of the force kernel is locally in $L^{p'}$, and its remaining part is
bounded. Hence~\eqref{eq:green} gives
\begin{equation}\label{eq:field-large-p}  \|E(t)\|_\infty\le C_p\bigl(\|f_0\|_1+\|\rho(t)\|_p\bigr),\qquad p>3.
\end{equation}

This estimate will be applied with $p=10/3$.
\end{remark}

\begin{proof}[Proof of Theorem~\ref{thm:moments}]
Corollary~\ref{cor:critical} gives~\eqref{eq:critical-main}. We prove moment propagation by increasing
the target order $k$.

For $2<k\le3$, we have $s(k)\le2$. Interpolating the $L^2$ density bound with the mass bound and
applying H\"older in time gives~\eqref{eq:density-condition}. Lemma~\ref{lem:single-moment} therefore
bounds $M_k$.

Next suppose $3<k\le5$. The initial third moment is finite, so the preceding case bounds $M_3$.
Lemma~\ref{lem:electric-source} then gives $\rho\in L^{22/9}_{t,x}$. Since

\[
 s(k)\le s(5)=\frac{24}{11}<\frac{22}{9},
\]
interpolation with the mass bound again yields~\eqref{eq:density-condition}, and
Lemma~\ref{lem:single-moment} propagates the target moment.

Finally, let $k>5$. The preceding cases propagate $M_5$, and Lemma~\ref{lem:electric-source} yields
$\rho\in L^{10/3}_{t,x}$. By~\eqref{eq:field-large-p}, this also gives~\eqref{eq:field-main}. Since
$s(k)<3<10/3$, interpolation with mass and H\"older in time imply~\eqref{eq:density-condition}. Hence
Lemma~\ref{lem:single-moment} propagates $M_k$. Each initial moment used in this
argument has order at most $k$.

Having propagated $M_k$ in all three cases, apply Lemma~\ref{lem:electric-source} once more, now with
$m=k$. Its first assertion applies for $k\le5$. For $k>5$, use the time-integrable bounded field just
obtained. The resulting density exponent is

\[
 2+\frac{4(k-2)}9=\frac{4k+10}{9}.
\]

When $k=5$, the same application of~\eqref{eq:field-large-p} proves~\eqref{eq:field-main}.
\end{proof}

\section{Global classical solutions}\label{sec:classical}

We now apply the density and moment estimates to maximal local classical solutions. 

\subsection{Velocity support}
We use the following potential estimates, with the second needed only for the Cauchy problem:
\begin{equation}\label{eq:field-split}
 \begin{aligned}
  \|E(t)\|_\infty
  &\le C\left(\|f_0\|_1+\|\rho(t)\|_2^{2/3}\|\rho(t)\|_\infty^{1/3}\right),\\
  \|\nabla E(t)\|_\infty
  &\le C\bigl(1+\|f_0\|_1+\|\rho(t)\|_\infty\bigr)
       \bigl[1+\log(e+\|\nabla\rho(t)\|_\infty)\bigr].
 \end{aligned}
\end{equation}
The first follows in both domains by splitting the Green integral in~\eqref{eq:green}
into near and far regions. The second follows from a decomposition of the differentiated
Coulomb kernel, using cancellation near the diagonal; see~\cite[Lemma P1]{Rein}.

For a classical solution with $\nu=0$, define
\[
 \cV(t)=1+\sup_{0\le s\le t}\sup\{|v|:(x,v)\in\supp f(s)\},
\]
Since $\rho(t,x)\le C\|f(t)\|_\infty\cV(t)^3$, it follows that
\begin{equation}\label{eq:linear-field}
 \|E(t)\|_\infty\le C_{\Tb}\left(1+\|\rho(t)\|_2^{2/3}\right)\cV(t).
\end{equation}
Along a characteristic $(X(t),V(t))$, where $V(t)$ denotes the particle velocity,
$V\cdot\cG=0$ and $\lambda\ge0$ give $\dd|V|/\dd t\le|E(t,X)|$.
Since specular reflection preserves speed, the same inequality holds for both $\mathcal O=\mathbb R^3$ and $\mathcal O=\Omega$. Combining it
with~\eqref{eq:linear-field}, Gr\"onwall's inequality, and H\"older in time gives
\begin{equation}\label{eq:support}
 \cV(T)\le \cV(0)\exp\left\{C_{\Tb}T+
 C_{\Tb}T^{2/3}\left(\int_0^T\|\rho(t)\|_2^2\dd t\right)^{1/3}\right\}.
\end{equation}

\begin{proof}[Proof of Theorem~\ref{thm:classical}\textup{(i)}] \emph{Step 1. Local solvability.}\enspace
The characteristic construction for local Vlasov--Poisson solutions~\cite{Rein} applies with the
prescribed $C^1$ drift $\cG-\lambda v$. Let $[0,T_*)$ be the maximal interval of existence. If the
velocity support is bounded on a finite interval, then trajectories issued from compact initial support
remain in a compact subset of phase space. The first derivatives of the prescribed drift are
consequently bounded on that set.

\smallskip\noindent\emph{Step 2. Continuation under a velocity bound.}\enspace
Assume that the velocity support is bounded. Then
$\|\nabla\rho\|_\infty\le C\cV(T)^3\|\nabla_{x,v}f\|_\infty$ for $t<T$.
Differentiating the transport equation and using the second inequality
in~\eqref{eq:field-split}, we obtain

\[
 D^+Z\le C_{\Tb,\cV(T)}Z\log Z,\qquad
 Z=e+\|\nabla_{x,v}f\|_\infty,
\]
where $D^+$ denotes the upper right Dini derivative. This inequality bounds $Z$ on finite intervals.
Hence the local solution can be continued as long as its velocity support remains bounded.

\smallskip\noindent\emph{Step 3. Global continuation.}\enspace
Lemma~\ref{lem:energy} and Corollary~\ref{cor:critical} bound the right-hand side of~\eqref{eq:support}
uniformly for $0<T<\min\{T_*,\Tb\}$. Step~2 therefore excludes a finite maximal lifespan. Local
uniqueness extends to every finite interval, and differentiation with elliptic regularity propagates
higher regularity. In the whole space, the velocity bound also controls the spatial support.
\end{proof}

\subsection{Velocity diffusion}
For $\nu>0$, polynomially weighted uniform estimates replace the velocity-support bound. We first prove
the weighted estimate, which is also valid at $\nu=0$.

\begin{lemma}[Polynomial velocity bounds]\label{lem:polynomial}
Let $f$ be a classical solution satisfying the kinetic maximum principle, and let $\ell\ge0$ with
$\vav^\ell f_0\in L^\infty$. If $E\in L^1(0,T;L^\infty)$, then
\begin{equation}\label{eq:polynomial}
 \|\vav^\ell f(t)\|_\infty
 \le\|\vav^\ell f_0\|_\infty
 \exp\left\{C_\ell(\nu+\lambda)t
                 +C_\ell\int_0^t\|E(s)\|_\infty\dd s\right\}.
\end{equation}
The estimate applies to every force satisfying~\eqref{eq:gyro}, with specular reflection when a
boundary is present.
\end{lemma}
\begin{proof}
Set $w=\vav^\ell$, $a=\nabla_v\log w$, and $g=wf$. Conjugating the kinetic equation by $w$ gives
\begin{equation*}
\begin{split}
 &\partial_tg+v\cdot\nabla_xg
 +(E+\cG-\lambda v+2\nu a)\cdot\nabla_vg-\nu\Delta_vg\\
 &\qquad=\bigl[3\lambda+(E+\cG-\lambda v)\cdot a
                +\nu(|a|^2-\operatorname{div}_va)\bigr]g.
\end{split}
\end{equation*}
The functions $a$ and $\nabla_va$ are bounded. Moreover, $\cG\cdot a=0$ and $-\lambda v\cdot a\le0$, so
the coefficient on the right is at most $C_\ell(\nu+\lambda+|E|)$. Since $w$ is radial, $g$ satisfies
the same specular condition as $f$. The kinetic maximum
principle therefore gives~\eqref{eq:polynomial}. The unbounded weight is justified by radial truncation,
or, for the local smooth solutions below, by applying the principle on bounded phase regions and passing
to the limit.
\end{proof}

\begin{proof}[Proof of Theorem~\ref{thm:classical}\textup{(ii)}] \emph{Step 1. Local construction.}\enspace
Local smooth solutions are obtained by iterating the Poisson field in polynomially weighted classical
norms, as in~\cite{Bouchut93,Bouchut95}. The additional drift $A(t)v$ is smooth and grows at most
linearly in velocity. We estimate the norms needed to repeat this construction at any finite time.

\smallskip\noindent\emph{Step 2. Density and field bounds.}\enspace
Lemma~\ref{lem:energy} controls $\|f\|_1$, $\|f\|_{\infty}$, and second moment on each local interval. Applying
Theorem~\ref{thm:moments} with $k=5$ therefore gives
\begin{equation}\label{eq:diffusive-input}
 \sup_{t<T}M_5(t)+\int_0^T\|\rho(t)\|_{10/3}^{10/3}\dd t
                   +\int_0^T\|E(t)\|_\infty\dd t\le C_{\Tb,M_5(0)}.
\end{equation}
The field integral in this estimate allows us to apply Lemma~\ref{lem:polynomial}. For any fixed
$\ell>3$, it follows that
\begin{equation}\label{eq:density-infinity}
 \sup_{t<T}\|\vav^\ell f(t)\|_\infty\le C_{\Tb,\ell},
 \qquad \sup_{t<T}\|\rho(t)\|_\infty\le C_{\Tb,\ell}.
\end{equation}
Together with the mass bound,~\eqref{eq:green} then bounds $E$ uniformly in time and space.

\smallskip\noindent\emph{Step 3. First derivatives.}\enspace
With the density and field now bounded, define
\[
 Z_1=e+\sum_{|\beta|\le1}\|\vav^\ell D_{x,v}^\beta f\|_\infty.
\]
Differentiate the kinetic equation, conjugate by the radial weight, and apply the maximum principle to a
smooth approximation of the norm of the derivative vector. The derivative of $A(t)v$ contributes the
bounded matrix $A(t)$, while the commutator with $v\cdot\nabla_x$ replaces a velocity derivative by a
spatial derivative of the same order. Consequently,
\begin{equation*}
 D^+Z_1\le C_\ell\bigl(1+\nu+\lambda+\|A(t)\|+\|E(t)\|_\infty
                                    +\|\nabla E(t)\|_\infty\bigr)Z_1.
\end{equation*}
Since $\|\nabla\rho\|_\infty\le C_\ell Z_1$, the second inequality in~\eqref{eq:field-split} and the
density bound~\eqref{eq:density-infinity} give $D^+Z_1\le C_{\Tb,A,f_0,\ell}Z_1\log Z_1$. Thus all weighted
first derivatives remain bounded on finite intervals.

\smallskip\noindent\emph{Step 4. Higher derivatives and polynomial decay.}\enspace
For $j\ge1$, the same kernel decomposition yields
\begin{equation*}
 \|D_x^jE\|_\infty
 \le C_j\|f_0\|_1+C_j(1+\|D_x^{j-1}\rho\|_\infty)
                      \bigl[1+\log(e+\|D_x^j\rho\|_\infty)\bigr].
\end{equation*}

Indeed, near the diagonal the first-derivative kernel acts on $D_x^{j-1}\rho$ with cancellation, while
away from the diagonal the derivatives fall on the smooth cutoff kernel and the integral is controlled
by mass.

Now let $j\ge2$ and assume that all lower-order weighted derivatives have been bounded. In the equation
for $D_{x,v}^jf$, the highest field derivative $D_x^jE$ is multiplied by a first velocity derivative of
$f$. The remaining terms are either linear in the order-$j$ norm or contain only lower-order
derivatives. Hence

\[
 D^+Z_j\le C_{\Tb,j,\ell}\bigl(1+Z_j+\log(e+Z_j)\bigr),\qquad
 Z_j=e+\sum_{|\beta|\le j}\|\vav^\ell D_{x,v}^\beta f\|_\infty.
\]
This inequality bounds $Z_j$, completing the induction.

To propagate spatial decay on $\R^6$, use the joint weight $w=(1+|x|^2+|v|^2)^N$. Its logarithmic
derivatives and $v\cdot\nabla_x\log w$ are bounded, while skew-symmetry gives $A(t)v\cdot\nabla_vw=0$.
The weighted equations therefore close in the same order, since the field derivatives have already been
bounded. This controls every fixed polynomially weighted derivative.

\smallskip\noindent\emph{Step 5. A uniform restart time.}\enspace
The local iteration requires only finitely many of these norms: a sufficiently large polynomial weight
and three bounded derivatives with a joint spatial--velocity weight. On a common short
interval, solve successive linear kinetic equations with the preceding Poisson field. The Poisson map is
Lipschitz from the corresponding weighted uniform norm to $L^\infty_x$; hence the difference equation is
contractive in a weaker weighted uniform norm after shortening the interval. The derivative bounds and
interpolation then give a classical limit.

Steps~2--4 give a common lifespan for restarting at $t_0<T_*$.
Taking $t_0\uparrow T_*$ excludes $T_*<\infty$.

\smallskip\noindent\emph{Step 6. Uniqueness and smoothness.}\enspace
Let $\delta f=f_1-f_2$ be the difference of two solutions with the same initial datum. Its equation
gives

\[
 \frac{\dd}{\dd t}\|\delta f\|_1\le\int|E_1-E_2|(x)q_2(x)\dd x,
 \qquad q_2(x)=\int|\nabla_v f_2(x,v)|\dd v.
\]
Diffusion contributes nonpositively. Friction and the prescribed drift are in divergence form and hence
give no $L^1$ growth. Splitting the kernel as in~\eqref{eq:green} and applying Fubini bounds the
remaining term by
\[
 C\bigl(\|q_2\|_\infty+\|q_2\|_1\bigr)\|\delta f\|_1.
\]
The weighted derivative estimates bound this coefficient on every finite interval. Gr\"onwall's
inequality gives uniqueness, and the differentiated equations yield the asserted smoothness.
\end{proof}

\begin{remark}[Spatially varying gyroscopic forces]\label{rem:diffusive-scope}
The form
$\cG=A(t)v$ also preserves the velocity weight under differentiation. By contrast, when $\cG$ depends on
$x$, the commutator $\partial_x\cG\cdot\nabla_vf$ may require an additional velocity weight. This is the
distinction between the force class in the diffusive classical theorem and the broader class in
Theorem~\ref{thm:moments}.
\end{remark}

\subsection{The specular--Dirichlet problem in convex domains}\label{subsec:convex}
Theorem~\ref{thm:boundary} uses the compatibility and grazing-flatness conditions
of~\cite[Theorem 2.1]{Cesbron}. We combine that local theory with the velocity bound just proved.

\begin{proof}[Proof of Theorem~\ref{thm:boundary}] \emph{Step 1. Uniform a priori bounds.}\enspace
Lemma~\ref{lem:energy} and Corollary~\ref{cor:critical}
control $\|f\|_1$, $\|f\|_{\infty}$, energy, and $L^2$ density norm on every finite interval. Since reflection
preserves speed,~\eqref{eq:support} bounds the velocity support uniformly for $0<T<\min\{T_*,\Tb\}$.

\smallskip\noindent\emph{Step 2. Boundary regularity and continuation.}\enspace
The local construction and regularity estimates of~\cite[Sections 4.1--4.2]{Cesbron} control the
classical norms and transport the flatness near grazing under a velocity bound. The required outward
normal force follows from the Hopf lemma. To make this sign uniform on a finite interval, note that the
velocity and $\|f\|_{\infty}$ bounds give a uniform bound for $\rho$. A sufficiently thin boundary layer
therefore contains less than half of the conserved mass. On the remaining compact interior set, the
Dirichlet Poisson kernel has a uniform positive lower bound. Its integral against $\rho$ gives a
uniform positive lower bound for $-\partial_nU$ on $\partial\Omega$.

The local construction consequently has a common restart time within the admissible class. Together
with Step~1, this excludes a finite maximal lifespan. Local uniqueness extends globally, and
Theorem~\ref{thm:moments} supplies the moment and density bounds.
\end{proof}

\begin{remark}[Velocity diffusion with specular--Dirichlet conditions]\label{rem:diffusive-boundary}
Theorem~\ref{thm:moments} also applies when $\nu>0$. If $M_5(0)<\infty$, it yields $E\in L^1(0,T;L^\infty(\Omega))$. Moreover, if $\langle v\rangle^\ell f_0\in L^\infty$ for some $\ell>3$, Lemma~\ref{lem:polynomial} gives a uniform bound on $\rho$ for classical solutions satisfying the kinetic maximum principle. Classical continuation additionally requires derivative estimates near the grazing set; see~\cite{ZhuKFP} for related results on linear kinetic Fokker--Planck equations.
\end{remark}

\end{document}